\documentclass[12pt,twoside]{amsart}\usepackage{amssymb}
\usepackage{enumerate}

\newtheorem{thmspec}{\relax}

\newtheorem{theorem}{Theorem}[section]

\newtheorem{lem}[theorem]{Lemma}
\newtheorem{cor}[theorem]{Corollary}
\newtheorem{prop}[theorem]{Proposition}

\newtheorem{rem}[theorem]{Remark}
\theoremstyle{definition}

\theoremstyle{remark}

\numberwithin{equation}{section}
\def \Bbb{\mathbb}

\def\onto{{\kern3pt\to\kern-8pt\to\kern3pt}}

\def\<{\langle}
\def\>{\rangle}
\def\|{{\ |\ }}

\def\onto{\twoheadrightarrow}

\def\-{\underline}

\def\inte{\operatorname{int}}
\def\N{\Bbb N}

\def\Q{\Bbb Q}
\def\R{\Bbb R}
\def\C{\Bbb C}

\def\X{\Bbb X}
\def\T{\Bbb T}

\def\<{\langle}
\def\>{\rangle}

\catcode`\@=11
\def\serieslogo@{\relax}
\def\@setcopyright{\relax}
\catcode`\@=12

\title[ Extension theorems of Sakai type ]
{  Extension theorems of Sakai type for\\
 separately holomorphic and meromorphic  functions
}

\begin{document}

\author{P. Pflug}
\address{Peter Pflug\\
Carl von  Ossietzky Universit\"{a}t Oldenburg \\
Fachbereich  Mathematik\\
Postfach 2503, D--26111\\
 Oldenburg, Germany}
\email{pflug@mathematik.uni-oldenburg.de}

\author{N. Vi\^et Anh}
\address{Nguy\^en Vi\^et Anh\\
Carl von  Ossietzky Universit\"{a}t Oldenburg \\
Fachbereich  Mathematik\\
Postfach 2503, D--26111\\
 Oldenburg, Germany}
\email{nguyen@mathematik.uni-oldenburg.de}

\date{}


\begin{abstract}
 We first exhibit counterexamples to some open questions related to  a theorem of Sakai.
 Then we establish an  extension theorem of Sakai type for separately
 holomorphic/meromorphic functions. 
\end{abstract}
\maketitle

\section{ Introduction }

We first fix some notations and terminology.
Throughout the paper,  $E$ denotes the unit disc of $\C,$ and for any set
$S\subset \C^n,$  $\inte S $ (or equivalently  $\inte_{\C^n}S $) denotes the
interior of $S.$ For any domain $D\subset \C^n,$ we  say that the subset $S\subset D$ does not separate
domains in $D$ if for every domain $U\subset D,$ the set
$U\setminus S$ is connected.
Moreover
$\mathcal{O}(D)$ (resp. $\mathcal{M}(D)$) will denote the space
of holomorphic (resp. meromorphic) functions on $D.$
Finally, if $S$ is a subset of $D\times G,$ where $D\subset \C^p,\ G\subset\C^q$ are some open
sets, then for $a\in D$ (resp. $b\in G$), the fiber $S(a,\cdot)$ (resp. $S(\cdot,b)$) is the
set $\lbrace w\in G:\ (a,w)\in S \rbrace$ (resp. $\lbrace z\in D:\ (z,b)\in S
\rbrace$).

In 1957 E. Sakai \cite{sa} claimed that he had proved the following result
\renewcommand{\thethmspec}{Theorem}
  \begin{thmspec}
 Let $S\subset E\times E$ be a relatively closed set such that
 $\inte S=\varnothing$ and $S$ does not separate
domains in $E\times E.$
Let $A$ (resp. $B$) be the set of all $a\in E$ (resp. $b\in E$)
such that $\inte_{\C} S(a,\cdot)=\varnothing$
(resp. $\inte_{\C} S(\cdot,b)=\varnothing$). Put
$X:=\X(A,B;E,E)=(A\times E)\cup (E\times B).$

Then for every function
$f:X\setminus S\longrightarrow \C$ which is separately meromorphic
on $X,$  there exists an $\hat{f}\in\mathcal{M}(E\times E)$
such that $\hat{f}=f$ on $X\setminus S.$
\end{thmspec}
Unfortunately, it turns out as reported in \cite{jp4} that
the proof of E. Sakai contains an essential gap. In the latter paper
 M. Jarnicki and the first author  also give a correct proof of this theorem.

E. Sakai also claimed in \cite{sa} that the following question (the $n$-dimensional version of the Theorem)
can be answered positively  but he did not
give any proof.
\renewcommand{\thethmspec}{Question 1}
  \begin{thmspec}
 For any $n\geq 3,$   let  $S\subset E^n$  be a relatively closed set
 such that
 $\inte S=\varnothing$ and $S$ does not separate domains.
Let $f:E^n\setminus S\longrightarrow \C$ be such that for any $j\in\left\lbrace 1,\ldots,n \right\rbrace$
and for any
  $(a^{'},a^{''})\in E^{j-1}\times E^{n-j},$
for which $\inte_{\C} S(a^{'},\cdot,a^{''}    )=\varnothing,$
the function $f(a^{'},\cdot,a^{''})$ extends
   meromorphically to $E.$
 Does $f$  always extend meromorphically to $E^n$ ?
\end{thmspec}

In connection with the Theorem and Question 1, M. Jarnicki and the first author \cite{jp4}
posed two more  questions~:
\renewcommand{\thethmspec}{Question 2}
  \begin{thmspec}
 Let $\mathcal{A}$ be a subset  of $E^n$ $(n\geq 2)$
 which is plurithin at $0\in E^n$ (see Section 2 below for the notion "plurithin").
 For an arbitrary open neighborhood $U$ of $0,$ does  there  exist a non-empty relatively open
 subset $C$ of a real hypersurface in $U$  such that $C\subset U\setminus
 \mathcal{A}$ ?
\end{thmspec}

\renewcommand{\thethmspec}{Question 3}
  \begin{thmspec}
 Let $D\subset\C^p,$ $G\subset \C^q$  $(p,q\geq 2)$ be pseudoconvex
 domains and let $S\subset D\times G$ be
    a relatively closed set such that  $\inte S=\varnothing$ and $S$ does not separate
    domains in $D\times G.$
Let $A$ (resp. $B$) be the set of all $a\in D$ (resp. $b\in G$)
such that $\inte_{\C^p} S(a,\cdot)=\varnothing$
(resp. $\inte_{\C^q} S(\cdot,b)=\varnothing$). Put
$X:=\X(A,B;D,G)=(A\times G)\cup (D\times B)$ and let
$f:X\setminus S\longrightarrow \C$ be a function which is separately
meromorphic
on $X.$ Does  there always exist a function $\hat{f}\in\mathcal{M}(D\times G)$
such that $\hat{f}=f$ on $X\setminus S$ ?
\end{thmspec}

This Note has two purposes.
The first one  is to give counterexamples to the three open
questions above.
The second one is to describe  the maximal domain  to which the function
$f$ in Questions 1 and 3 can  be  meromorphically extended.

This paper is organized as follows.

We begin Section 2 by
collecting some background of the pluripotential theory
and introducing some notations. This preparatory  is necessary  for us  to
 state the results afterwards.

 Section 3 provides three counterexamples to the three open questions
 from above.

The subsequent sections are devoted to the proof of a result in the  positive
direction. More
 precisely,  we describe qualitatively the maximal domain of meromorphic extension of  the function
$f$ in Questions 1 and 3.
  Section 4 develops auxiliary tools that will be used in Section 5  to
  prove the  positive result.
%
%
%

\smallskip

\indent{\it{\bf Acknowledgment.}} The paper was written while  the second author was visiting the
Carl von  Ossietzky Universit\"{a}t Oldenburg
 being  supported by The Alexander von Humboldt
Foundation. He wishes to express his gratitude to these organisations.

\section{ Background and Statement of the results}
We keep the main notation from \cite{jp4}.

 Let $N\in\N,\ N\geq 2,$ and let $\varnothing\not = A_j\subset
D_j\subset\C^{n_j},$ where $D_j$ is a domain, $j=1,\ldots,N.$ We define
an {\it $N$-fold cross}
\begin{eqnarray*}
X &=&\X(A_1,\ldots,A_N; D_1,\ldots,D_N)\\
&:=& \bigcup_{j=1}^N  A_1\times\cdots\times A_{j-1}\times D_j\times
A_{j+1}\times\cdots A_N\subset \C^{n_1+\cdots+n_N} =\C^n.
\end{eqnarray*}

 For an open set $\Omega\subset\C^n$ and $A\subset \Omega,$ put
\begin{equation*}
h_{A,\Omega}:=\sup\left\lbrace u:\ u\in\mathcal{PSH}(\Omega),\ u\leq 1 \ \text{on}\
\Omega, \ u\leq 0\ \text{on}\ A \right\rbrace,
\end{equation*}
where $\mathcal{PSH}(\Omega)$ denotes the set of all plurisubharmonic
functions on $\Omega.$ Put
\begin{equation*}
\omega_{A,\Omega}:=\lim\limits_{k\to +\infty} h^{\ast}_{A\cap
\Omega_k,\Omega_k},
\end{equation*}
where $\left\lbrace \Omega_k \right\rbrace^{\infty}_{k=1}$ is a sequence
of relatively compact open sets $\Omega_k\subset \Omega_{k+1}\Subset
\Omega$ with $\cup_{k=1}^{\infty} \Omega_k=\Omega$
($h^{\ast}$ denotes the upper semicontinuous regularization of $h$).
 We say that a subset $\varnothing\not= A\subset \C^n$ is {\it
locally pluriregular} if $h^{\ast}_{A\cap\Omega,\Omega}(a)=0$
for any $a\in A$ and for any open neighborhood $\Omega$ of $a.$
We say that $A$ is {\it plurithin} at a point $a\in\C^n$ if either
$a\not\in\overline{A}$ or $a\in\overline{A}$ and
$\limsup_{A\setminus{\lbrace a\rbrace}\owns z\to a} u(z)<u(a)$ for a
suitable
function $u$ plurisubharmonic in a neighborhood of $a.$
For a good background of the pluripotential
theory, see the books \cite{kl} or   \cite{jp1}.

 For an $N$-fold cross $X:=\X(A_1,\ldots,A_N;D_1,\ldots,D_N)$ let
\begin{equation*}
\widehat{X}:=\left\lbrace \left(z_1,\ldots,z_N\right )\in D_1\times
\cdots\times D_N:\ \sum\limits_{j=1}^N \omega_{A_j,D_j}(z_j)<1
\right\rbrace.
\end{equation*}

 Suppose that $S_j\subset (A_1\times\cdots\times
 A_{j-1})\times(A_{j+1}\times \cdots \times A_N),\ j=1,\ldots,N.$
Define the {\it generalized $N$-fold cross}
\begin{eqnarray*}
T&=&\T(A_1,\ldots,A_N;D_1,\ldots,D_N;S_1,\ldots,S_N):=\bigcup_{j=1}^N
\left\lbrace (z^{'},z_j,z^{''}) \right.\\
&&\qquad\left. \in (A_1\times\cdots\times
 A_{j-1})\times D_j\times (A_{j+1}\times \cdots \times A_N):\
 (z^{'},z^{''})\not\in S_j\right\rbrace.
\end{eqnarray*}

Let $M\subset T$ be a relatively closed set. We say that a function $f:
T\setminus M\to \C$  (resp.  $f:
(T\setminus M)\setminus S\to \C$)  is {\it separately holomorphic} and
write $f\in\mathcal{O}_s(T\setminus M)$ (resp. {\it separately meromorphic}
and
write $f\in\mathcal{M}_s(T\setminus M)$) if for any $j\in\lbrace
1,\ldots,N\rbrace$ and $(a^{'},a^{''})\in  (A_1\times\cdots\times
 A_{j-1})\times(A_{j+1}\times \cdots \times A_N)\setminus S_j$
 the function $f(a^{'},\cdot,a^{''})$ is holomorphic on (resp. can be   meromorphically extended to)
  the open set
 $D_j\setminus M(a^{'},\cdot,a^{''}),$
 where $M(a^{'},\cdot,a^{''}):=\left\lbrace z_j\in\C^{n_j}:\
 (a^{'},z_j,a^{''})\in M\right\rbrace.$

We are now ready to state the results. The following propositions
give  negative answers to Questions 2, 3 and 1 respectively.
\renewcommand{\thethmspec}{Proposition A}
  \begin{thmspec}
 For any $n\geq 2,$ there is an open dense subset $\mathcal{A}$ of $E^n$
 which is plurithin at $0$ and there exists no non-empty relatively open
 subset $C$ of a real hypersurface such that $C\subset E^n\setminus
 \mathcal{A}.$
\end{thmspec}
\renewcommand{\thethmspec}{Proposition B}
  \begin{thmspec}
 Let $D\subset\C^p,$ $G\subset \C^q$  $(p,q\geq 2)$ be pseudoconvex
 domains. Then there is a relatively closed set $S\subset D\times G$ with
 the following properties
\begin{itemize}
\item[$(i)$]  $\inte S=\varnothing$ and $S$ does not separate domains;
\item[$(ii)$]
let $A$ (resp. $B$) be the set of all $a\in D$ (resp. $b\in G$)
such that $\inte_{\C^p} S(a,\cdot)=\varnothing$
(resp. $\inte_{\C^q} S(\cdot,b)=\varnothing$) and put
$X:=\X(A,B;D,G),$ then there exists a  function
$f:X\setminus S\longrightarrow \C$ which is separately holomorphic
on $X$ and there is no function $\hat{f}\in\mathcal{M}(D\times G)$
such that $\hat{f}=f$ on $X\setminus S.$
\end{itemize}
\end{thmspec}
\renewcommand{\thethmspec}{Proposition C}
  \begin{thmspec}
  For all $n\geq 3,$  there is a relatively closed set $S\subset E^n$ with
 the following properties
\begin{itemize}
\item[$(i)$]  $\inte S=\varnothing$ and $S$ does not separate domains;
\item[$(ii)$] for $1\leq j\leq n,$
let  $S_j$ denote the set of all $(a^{'},a^{''})\in E^{j-1}\times E^{n-j}$
such that $\inte_{\C} S(a^{'},\cdot,a^{''}    )\not=\varnothing$
 and define the $n$-fold generalized cross
$T:=\T(E,\ldots,E;E,\ldots,E;S_1,\ldots,S_n),$ then there is a  function
$f:T\setminus S\longrightarrow \C$ which is separately holomorphic
on $T$ and there is no function $\hat{f}\in\mathcal{M}(E^n)$
such that $\hat{f}=f$ on $T\setminus S.$
\end{itemize}
\end{thmspec}

\renewcommand{\thethmspec}{Problem}
  \begin{thmspec}
  Are the answers to Questions 1 and 3 positive if the condition on $S$
  is sharpened in the following form: $S$ does not separate
  lower dimensional domains?
\end{thmspec}

Finally, we  state a result in positive direction.
\renewcommand{\thethmspec}{Theorem D}
  \begin{thmspec}
  For all $j\in\left\lbrace 1,\ldots,N\right\rbrace\ \ (N\geq 2),$ let $D_j$ be a pseudoconvex domain in $\C^{n_j}$
  and let $S$ be a relatively closed set of $D:=D_1\times\cdots\times D_N$ with
    $\inte S=\varnothing.$
     For $j\in\left\lbrace 1,\ldots,N\right\rbrace,         $
let  $S_j$ denote the set of all $(a^{'},a^{''})\in (D_1\times\cdots\times
 D_{j-1})\times(D_{j+1}\times \cdots \times D_N) $
such that $\inte_{\C^{n_j}} S(a^{'},\cdot,a^{''}    )\not=\varnothing$
 and define the $N$-fold generalized cross
$T:=\T(D_1,\ldots,D_N;D_1,\ldots,D_N;S_1,\ldots,S_N).$  Let $f\in \mathcal{O}_s(T\setminus S)$ (resp. $f\in
\mathcal{M}_s(T\setminus S)$).
\begin{itemize}
\item[$(i)$]
Then there are an open dense set $\Omega$ of $D$
and exactly one function
$\hat{f}\in    \mathcal{O}(\Omega)$ 
 such that
 $\hat{f}=f$ on $(T\cap\Omega)\setminus S.$
\item[$(ii)$] In the case where $N=2,$ (i) can be strengthened as follows.
Let $\Omega_j$ be a relatively compact pseudoconvex subdomain of $D_j$
(j=1,2). Then there are an open dense set $A_j$ in $\Omega_j$
and exactly one function
$\hat{f}\in    \mathcal{O}(\widehat{X})$ (resp. $\hat{f}\in    \mathcal{M}(\widehat{X})$),
where $X:=\X(A_1,A_2;\Omega_1,\Omega_2),$ such that
 $\hat{f}=f$ on $(T\cap\widehat{X})\setminus S.$
\end{itemize}
\end{thmspec}

A remark is in order. In contrast with the other usual extension theorems
(see  \cite{jp1}, \cite{jp2}, \cite{jp3}, \cite{jp4} and the references therein), the domain
of
meromorphic/holomorphic extension of the function $f$ in Theorem D depends
on $f.$

\section{Three counterexamples}
In the sequel we will fix a function $v\in\mathcal{SH}(2E)$ such that
$v(0)=0$ and the complete polar set $\left\lbrace z\in 2E:\
v=-\infty\right\rbrace$
is dense in $2E.$ For example one can choose $v$ of the form
\begin{equation}
v(z):=\sum\limits_{k=1}^{\infty} \frac{\log{\left(\frac{\vert z-q_k \vert}{4}\right )}}{d_k} -
\sum\limits_{k=1}^{\infty} \frac{\log{\left(\frac{\vert q_k \vert}{4}\right)}}{d_k} ,
\end{equation}
where $(\Q+i\Q)\cap 2E=\left\lbrace q_1,q_2,\ldots,q_k,\ldots\right\rbrace,$ and
$\left\lbrace d_k \right\rbrace_{k=1}^{\infty}$ is any sequence of positive
real numbers such that
$\sum\limits_{k=1}^{\infty} \frac{\log{\left(\frac{\vert q_k \vert}{4}\right)}}{d_k}
$ is finite.

For any positive integer $n\geq 2,$ define a new function $u\in\mathcal{PSH}((2E)^n)$
 and a subset $\mathcal{A}$ of $E^n$ as follows
\begin{equation} \label{e:barwq}
\begin{split}
u(z)&:= \sum\limits_{k=1}^n v(z_k),\qquad z= (z_1,\ldots,z_n)\in
(2E)^n,\\
\mathcal{A}&=\mathcal{A}_n:=\left\lbrace z\in E^n:\ u(z)<-1  \right\rbrace.
\end{split}
\end{equation}

Observe that  $\mathcal{A}$ is an open dense set of $E^n$ because  $\mathcal{A}$
contains the set  $\left\lbrace z\in E:\
v=-\infty\right\rbrace\times \cdots\times \left\lbrace z\in E:\
v=-\infty\right\rbrace$ which is dense in $E^n$ by our construction (3.1) above.

\begin{prop}
Let $\mathcal{S}$ be any closed set contained in the closed set
$E^n\setminus \mathcal{A}.$ Then $\mathcal{S}$ does not separate domains.
\end{prop}
Taking this proposition for granted, we are now able to complete the
proof
of Proposition A.
\\
\noindent{\it Proof of Proposition A.}

It is clear from (3.1) and (3.2) that the open dense set $\mathcal{A}$ is plurithin
at $0\in E^n.$ By Proposition 3.1, the closed set $E^n\setminus \mathcal{A}$
does not separate domains. Therefore this set cannot contain any open set
of a real hypersurface. Thus $\mathcal{A}$ has all the desired properties.
\hfill $\square$

We now come back to Proposition 3.1.
\begin{proof}
One first observe that
\begin{equation*}
\mathcal{S}\subset E^n\setminus \mathcal{A}= \left\lbrace z\in E^n:\
u(z)\geq -1\right\rbrace.
\end{equation*}

For any tuple of four vectors in $\R^n$ $a:=(a_1,\ldots,a_n),$
 $b:=(b_1,\ldots,b_n),$ $c:=(c_1,\ldots,c_n),$ $d:=(d_1,\ldots,d_n)$
 with the property that $a_k<b_k$ and $c_k<d_k$ for all $k=1,\ldots,n,$
 one defines the open cube of $\C^n$
 \begin{equation*}
\Delta=\Delta(a,b,c,d):=\left\lbrace z\in\C^n:\ a_k< \text{Re}z_k <b_k,\
 c_k< \text{Im}z_k <d_k,\  k=1,\ldots,n    \right\rbrace.
\end{equation*}
It is clear that the intersection of two such cubes is either empty or
a cube.

One first  shows that for any cube $\Delta\subset E^n$ the open set  $\Delta\setminus\mathcal{S}$
is connected. Indeed, pick two points  $z=(z_1,\ldots,z_n)$ and
$w=(w_1,\ldots,w_n)$ in $\Delta\setminus \mathcal{S}.$ Since
$\left\lbrace z\in E:\ v(z)=-\infty\right\rbrace$ is dense in $E,$ we can
choose $z^{'}=(z^{'}_1,\ldots,z^{'}_n)$ and
$w^{'}=(w^{'}_1,\ldots,w^{'}_n)$ in $\Delta\setminus \mathcal{S}$ such
that
\begin{itemize}
\item[$(i)$]
the segments $\gamma_1(t):=(1-t)z+tz^{'}$ and
$\gamma_3(t):=(1-t)w+tw^{'},$ $0\leq t\leq 1,$
are contained in $\Delta\setminus \mathcal{S};$
\item[$(ii)$]  $z^{'}_1,\ldots,z^{'}_n$ and
$w^{'}_1,\ldots,w^{'}_n$ are in $\left\lbrace z\in E:\ v(z)=-\infty\right\rbrace.$
\end{itemize}

Consider now  $\gamma_2:[0,1]\longrightarrow \Delta $
given by
\begin{equation*}
\gamma_2(t):=\left ( w^{'}_1,\ldots,w^{'}_{j}, (j+1-nt)z^{'}_{j+1}+(nt-j)w^{'}_{j+1},
z^{'}_{j+2},\ldots,z^{'}_n      \right ),
\end{equation*}
for $ t\in
\left \lbrack\frac{j}{n},\frac{j+1}{n}\right\rbrack$ and $
j=0,..,n-1.$ By (3.2) and property (ii) above,
$\gamma_2(t)\in \left\lbrace z\in E^n:\ u(z)=-\infty\right\rbrace$
for all $t\in[0,1].$ This implies that
 $\gamma_2:[0,1]\longrightarrow \Delta\setminus \mathcal{S}.$

Observe that $\gamma_2(0)=z^{'}$ and $\gamma_2(1)=w^{'}.$  By virtue of (i), the new
path  $\gamma:[0,1]\longrightarrow \Delta\setminus \mathcal{S}$ given by
\begin{align*}
\gamma(t):= \left\lbrace
\begin{array}{l}
\gamma_1(3t),\qquad \qquad t\in \left \lbrack 0,\frac{1}{3}\right\rbrack,\\
\gamma_2(3t-1),\qquad  t\in\left \lbrack \frac{1}{3},\frac{2}{3}
\right\rbrack,\\
\gamma_3(3t-2),\qquad  t\in\left \lbrack \frac{2}{3},1 \right\rbrack.
\end{array}
\right.
\end{align*}
satisfies  $\gamma(0)=z$ and $\gamma(1)=w,$ and  $\Delta\setminus\mathcal{S}$
is therefore connected.

Now let $U$ be any subdomain of  $E^n.$ We wish to show that $U\setminus \mathcal{S}$
is connected. To do this, pick  points $z=(z_1,\ldots,z_n)$ and
$w=(w_1,\ldots,w_n)$ in $U\setminus \mathcal{S}.$ Since $U$ is arcwise
connected, there is a continuous function $\gamma: [0,1]\longrightarrow U$
such that $\gamma(0)=z$ and $\gamma(1)=w.$

By the Heine-Borel Theorem, the compact set  $\mathcal{L}:=\gamma([0,1])$ can
be covered  by a finite number of cubes $\Delta_l$ $(1\leq l\leq N)$ with
$\Delta_l\subset U$ and $\Delta_l\cap \mathcal{L}\not=\varnothing.$ Since  the path $\mathcal{L}$ is connected, the union
$\bigcup_{l=1}^N \Delta_l$ is also connected.

Suppose without loss of generality that $z\in\Delta_1$ and $w\in\Delta_N.$
From the discussion above, if $\Delta_1\cap\Delta_2\not =\varnothing$ then
$(\Delta_1\setminus \mathcal{S})\cap (\Delta_2\setminus \mathcal{S})=
(\Delta_1\cap\Delta_2)\setminus \mathcal{S}$ is  connected, and hence
$(\Delta_1\cup\Delta_2)\setminus \mathcal{S}$ is also connected.
Repeating this argument at most $N$ times and using the connectivity of
$\bigcup_{l=1}^N \Delta_l,$ we finally conclude that
$\bigcup_{l=1}^N \Delta_l \setminus \mathcal{S} (\subset U\setminus \mathcal{S})$
is also connected. This completes the proof.
\end{proof}
\begin{cor}
\begin{itemize}
\item[$(i)$] If $\mathcal{S}_1,\ldots,\mathcal{S}_N$ are relatively closed
subsets of $E^n$ which do not separate domains, then
the union $\bigcup_{l=1}^N \mathcal{S}_l$  does not separate domains too.
\item[$(ii)$] Let $\mathcal{A},\mathcal{S}$ be as in Proposition 3.1.
Then for any closed sets $F_1$ in $\C^p$ and $F_2$ in $\C^q$ $(p,q\geq 0),$  the closed set $F_1\times\mathcal{S}\times
F_2$  does not separate domains in $\C^p\times E^n\times \C^q.$
\end{itemize}
\end{cor}
\begin{proof}
To prove part (i), let $U$ be any subdomain of $E^n.$
Since $U\setminus\left(\bigcup_{l=1}^N \mathcal{S}_l \right)=
\left(\left (U\setminus\mathcal{S}_1
\right)\cdots\setminus\mathcal{S}_N\right),$
part (i) follows from the hypothesis of $\mathcal{S}_l.$

To prove part (ii), consider any subdomain $U$ of  $\C^p\times E^n\times \C^q$
and let $(z_1,w_1,t_1),(z_2,w_2,t_2)$ be two points in $ U\setminus( F_1\times\mathcal{S}\times
F_2).$ Since  $\mathcal{A} $ is an open dense set of $E^n,$ $\inte(  F_1\times\mathcal{S}\times
F_2     )=\varnothing,$ and therefore we are able to perform the compact
argument that we had already used in the proof of Proposition 3.1. Consequently, one is reduced to
the case where $U$ is a cube of $\C^{p+n+q}.$

Another reduction is in order.  Since $U\setminus (F_1\times\mathcal{S}\times
F_2)$ is open and  $\mathcal{A} $ is  dense in $E^n,$ by replacing
$w_1$ (resp. $w_2$) by  $w^{'}_1$  (resp. $w^{'}_2$) close to $w_1$ (resp.
$w_2$),
we may suppose  that $w_1,w_2\in
E^n\setminus\mathcal{S}.$

Write the cube $U$ as the product of
$\Delta_1\times\Delta_2\times\Delta_3,$
where $\Delta_1$ (resp. $\Delta_2$ and $\Delta_3$) is a cube in $\C^p$
(resp. $\C^n$ and $\C^q$). By Proposition 3.1,
there is a continuous path $\gamma_2:[0,1]\longrightarrow \Delta_2\setminus \mathcal{S}$
such that $\gamma_2(0)=w_1$ and $\gamma_2(1)=w_2.$

We now consider the path $\gamma:[0,1]\longrightarrow \Delta_1\times\Delta_2\times\Delta_3
\setminus \mathcal{S},$ where
$\gamma(t):=\left( \gamma_1(t),\gamma_2(t),\gamma_3(t)\right)$ and
$\gamma_1(t):=(1-t)z_1+tz_2, $
$\gamma_3(t):=(1-t)t_1+tt_2, $  $t\in [0,1].$
It easy to see that $\gamma(0)=(z_1,w_1,t_1)$ and $\gamma(1)=(z_2,w_2,t_2) ,$
which finishes the proof.
\end{proof}
The following two lemmas will be crucial for the proof of Propositions B and
C.
\begin{lem}
For an  open set $\Omega\subset \C^n$ and $A\subset \Omega,$ we have
either $\omega_{A,\Omega}\equiv 0$ or $\sup_{\Omega}\omega_{A,\Omega}=1.$
\end{lem}
\begin{proof} We first prove the lemma in the case where $\Omega$ is
bounded.
Suppose in order to get a contradiction that
$\sup_{\Omega}h^{\ast}_{A,\Omega}=M$ with $0<M<1.$
By virtue of the definition of  $h^{\ast}_{A,\Omega},$
it follows that
\begin{multline*}
\left\lbrace u:\ u\in\mathcal{PSH}(\Omega),\ u\leq 1 \ \text{on}\
\Omega,  u\leq 0\ \text{on}\ A \right\rbrace\\
=\left\lbrace u:\ u\in\mathcal{PSH}(\Omega),\ u\leq M \ \text{on}\
\Omega,  u\leq 0\ \text{on}\ A \right\rbrace.
\end{multline*}
 Therefore, $ h^{\ast}_{A,\Omega} (z)< M h^{\ast}_{A,\Omega} <h^{\ast}_{A,\Omega}         $ for any $z\in\Omega$
 with $     h^{\ast}_{A,\Omega} (z)>0,$ and we obtain the desired
contradiction.

The general case is analogous using the definition of $\omega_{A,\Omega}$
and the Hartog's Lemma.
\end{proof}

\begin{lem}
Let $\Omega_1\subsetneq \Omega_2$ be two domains of $\C^n$ such that
$\Omega_2$ is pseudoconvex. Assume that
there is a upper bounded function $\phi\in\mathcal{PSH}(\Omega_2)$ satisfying
$\Omega_1=\left\lbrace z\in\Omega_2:\ \phi(z)<0\right\rbrace.$
Then there is a function $f\in\mathcal{O}(\Omega_1)$ such that
there is no function $\hat{f}\in\mathcal{M}(\Omega_2)$ verifying
$\hat{f}=f$ on $\Omega_1.$
\end{lem}
\begin{proof}
It is clear from the hypothesis that $\Omega_1$ is also pseudoconvex.
Let $\partial \Omega_1$ be the boundary of $\Omega_1$ in $\Omega_2$
and let $S$ be a countable dense subset of $\partial \Omega_1.$
It is a classical fact that there is a function $f\in\mathcal{O}(\Omega_1)$ such that
\begin{equation}
\lim\limits_{z\in\Omega_1,\ z\to w}\vert f(z)\vert=\infty,\qquad\ w\in
S.
\end{equation}
 We will show that this is the desired function. Indeed, suppose in
order to get a contradiction that there is a function $\hat{f}\in\mathcal{M}(\Omega_2)$ verifying
$\hat{f}=f$ on $\Omega_1.$
Because of (3.3), $S$ and then $\partial\Omega_1$ are contained in the
pole set of $\hat{f}$ (i.e.  the union of the set of all poles of $\hat{f}$
and the set of all indeterminancy points of $\hat{f}$). Therefore, for any point  $w\in\partial\Omega_1,$ there is a small
open neighborhood $U$  of $w$ and a complex analytic subset of codimension one $C$ such that
$U\setminus C\subset \Omega_1.$ Since $\phi\in\mathcal{PSH}(U)$ is upper
bounded, $\phi(w)=\limsup\limits_{z\in U\setminus C,\ z\to w} \phi(z)=\phi(w) \leq
0$ for all $w\in\partial\Omega_1.$ Since  $\Omega_1\subsetneq \Omega_2,$ $\phi$ is non-constant
and therefore $\phi(w)<0$ for all $w\in\partial\Omega_1,$  which is a
contradiction.
\end{proof}

We are now ready to prove Propositions B and C.
\\
{\it The proof of Proposition B.}
Suppose, without loss of generality, that
$D=E^p$ and $G=E^q.$ The general case is almost analogous.
Let $F_p$ (resp. $F_q$) be any closed ball contained in the open set
$\mathcal{A}_p$ (resp. $\mathcal{A}_q$). We now define the relatively
closed set $S$ by the formula
\begin{equation}
S:= (E^p\setminus \mathcal{A}_p)\times F_q \bigcup F_p\times  (E^q\setminus
\mathcal{A}_q).
\end{equation}
We now check the properties  (i) and (ii) of Proposition B. First, $\inte S=\varnothing$
because $ \mathcal{A}_p$ (resp.  $\mathcal{A}_q$) is open dense set in
$E^p$ (resp. $E^q$). Second, by Proposition 3.1 and Corollary 3.2(ii),
the two relatively closed sets $ (E^p\setminus \mathcal{A}_p)\times F_q$
 and $F_p\times  (E^q\setminus
\mathcal{A}_q)$ do not separate domains. By Corollary 3.2(i), the union
$S$ also enjoys this property. Thus $S$ satisfies (i).

Using (3.4), a direct computation gives that
$A=\mathcal{A}_p$ and $B=\mathcal{A}_q$ and $A,B$ are open, in
particular they are
locally pluriregular.

By the classical cross theorem (see for instance \cite{nz} or \cite{jp1}), the envelope
 of holomorphy of $X$ is given by
\begin{equation*}
\widehat{X}:=\left\lbrace (z,w)\in E^p\times E^q:\qquad \omega_{A,E^p}
(z)+\omega_{B,E^q}(w) <1\right\rbrace.
\end{equation*}

We now show that $h^{\ast}_{\mathcal{A}_n,E^n}(0)>0$ for $n\geq 2.$
Indeed, let $M:=\sup_{E^n} u,$ where $u$ is defined in (3.2). Observe that $M>0$ since $u(0)=0.$ Consider the
function  $\tilde{u} \in\mathcal{PSH}(E^n)$ given by
\begin{equation*}
\tilde{u}(z):=\frac{u(z)-M}{M+\frac{1}{2}}+1, \qquad\text{for}\ z\in E^n.
\end{equation*}
It can be easily checked that $\tilde{u}(z)\leq 1$ on $E^n$ and $\tilde{u}(z)\leq 0$ on
$\mathcal{A}_n.$
Thus $\tilde{u}(0)\leq h^{\ast}_{\mathcal{A}_n,E^n}(0).$ On the other hand,
$\tilde{u}(0)=\frac{1}{2M+1}>0.$ Hence our assertion above follows.

We next show that
$ \widehat{X}\subsetneq E^p\times E^q.$
Indeed, we have
\begin{equation*}
\left\lbrace  w\in  E^q:\  (0,w)\in \widehat{X}   \right\rbrace \subset
\left\lbrace  w\in  E^q:\  h^{\ast}_{\mathcal{A}_q,E^q}(w)< 1-h^{\ast}_{\mathcal{A}_p,E^p}(0)
\right\rbrace.
\end{equation*}
Since $ h^{\ast}_{\mathcal{A}_q,E^q}(w)>0$ and  $h^{\ast}_{\mathcal{A}_p,E^p}(0)>0,$
Lemma 3.3 applies and consequently the latter set is strictly contained in
$E^q.$ This proves our assertion above.

We are now ready to complete the proof. By Lemma 3.4, there is a
holomorphic function $f$ in $ \widehat{X}$ which cannot be
meromorphically  extended to $E^p\times E^q.$ Therefore, there is no meromorphic
function $\hat{f}\in\mathcal{M}(E^p\times E^q)$ such that
$\hat{f}=f$ on the set of unicity for meromorphic functions
\begin{equation*}
X\setminus S=\left((\mathcal{A}_p\setminus F_p)\times E^q\right)\cup
(F_p\times \mathcal{A}_q)\cup \left(E^p\times (\mathcal{A}_q\setminus F_q) \right)
\cup(\mathcal{A}_p\times F_q).
\end{equation*}
 The proof is thereby finished.
\hfill $\square$

\smallskip

\noindent{\it The proof of Proposition C.}
In order to simplify the notation, we only consider the case $n=3,$
the general case $n> 3$ is analogous.
Let $\mathcal{B}$ be the following open dense subset of $E$
\begin{equation*}
\mathcal{B}:=\left\lbrace  z\in  E:\quad v(z)<-\frac{1}{2} \right\rbrace,
\end{equation*}
where $v$ is given by (3.1). Then by virtue of (3.2), it can be checked
that $(E\setminus \mathcal{B})\times (E\setminus \mathcal{B})\subset
E^2\setminus \mathcal{A}_2.$ Fix any closed ball $F$  contained in the open set
$\mathcal{B}.$ Next on applies Proposition 3.1 and Corollary 3.2 to the
relatively
closed set $\mathcal{S}:=(E\setminus \mathcal{B})\times (E\setminus
\mathcal{B}).$
Consequently, the  set
\begin{equation}
S:=\left(F\times(E\setminus \mathcal{B})\times (E\setminus \mathcal{B})\right)\cup
\left((E\setminus \mathcal{B})\times F\times (E\setminus
\mathcal{B})\right)
\cup \left((E\setminus \mathcal{B})\times (E\setminus \mathcal{B})\times
F\right)
\end{equation}
does not separate domains in $E^3.$ Moreover, since $\mathcal{B}$ is an
open dense subset of $E,$
we see  that $\inte S=\varnothing$ and $S$ is relatively closed.
Hence $S$ satisfies property (i).

To verify (ii), one first computes  the following sets using (3.5)
\begin{equation}
S_1=S_2=S_3=(E\setminus \mathcal{B})\times (E\setminus \mathcal{B}),\
T=\left(\mathcal{B}\times E\times E\right)\cup \left(E\times \mathcal{B}\times E\right)\cup
\left (E\times E\times \mathcal{B}\right).
\end{equation}
Next, by the product property for the relative extremal function \cite{ns},
we have $h^{\ast}_{\mathcal{B}\times\mathcal{B}, E^2}(0)=
h^{\ast}_{\mathcal{B}, E}(0).$ Since $\mathcal{B}\times\mathcal{B}\subset
\mathcal{A}_2$ and we have shown in Proposition B that $ h^{\ast}_{\mathcal{A}_2, E^2}(0)>0,$  it follows
 that $h^{\ast}_{\mathcal{B}, E}(0)>0.$

Consider now the domain of holomorphy
\begin{equation}
\Omega:=\left\lbrace (z,w,t)\in E^3:\qquad
h^{\ast}_{\mathcal{B}, E}(z)+h^{\ast}_{\mathcal{B}, E}(w)+h^{\ast}_{\mathcal{B},
E}(t)<2\right\rbrace.
\end{equation}
Since $\mathcal{B}$ is open and therefore locally pluriregular,
 it can be proved using Lemma 5 in \cite{jp2} that $\Omega$ is a domain.
Moreover it can be easily checked that $T\subset\Omega$ using (3.6) and
(3.7).

We now prove that $\Omega\subsetneq E^3.$ Indeed, since $h^{\ast}_{\mathcal{B}, E}(0)>0,$
by Lemma 3.3 there are $z,w\in E$ such that
 $h^{\ast}_{\mathcal{B}, E}(z)>\frac{2}{3},$ $h^{\ast}_{\mathcal{B}, E}(w)>\frac{2}{3}.$
 Then the fiber
 \begin{equation*}
\left\lbrace t\in E:\quad (z,w,t)\in\Omega \right\rbrace \subset
\left\lbrace t\in E: \quad h^{\ast}_{\mathcal{B}, E}(t)<\frac{2}{3}
\right\rbrace.
 \end{equation*}
 Another application of Lemma 3.3 shows that the latter set is strictly
 contained in $E.$ This proves our assertion from above.

We are now ready to complete the proof. By Lemma 3.4, there is a
holomorphic function $f$ in $ \Omega$ which cannot be
meromorphically extended to $E^n.$ Therefore, there is no meromorphic
function $\hat{f}\in\mathcal{M}(E^n)$ such that
$\hat{f}=f$ on the set of unicity for meromorphic functions
$T\setminus S.$ Hence, the proof is  finished.
\hfill $\square$

\section{Auxiliary results}
%
%
%


Let  $S$ be a subset  of an open set $D\subset \C^n.$ Then $S$ is said to be of {\it  Baire category I} if
 $S$ is contained in  a countable  union of
relatively closed sets in $D$ with empty interior.  Otherwise,
$S$   is said to be of {\it Baire category
II}.

The following lemma is very useful.
\begin{lem}
 For  $j\in\left\lbrace 1,\ldots,M\right\rbrace$ and $M\geq 2,$
  let $\Omega_j$ be a  domain in $\C^{m_j}$
  and let $S$ be a relatively closed set of $\Omega_1\times\cdots\times \Omega_M$ with
    $\inte S=\varnothing.$
 For $a_j\in \Omega_j, $ $j\in\left\lbrace 3,\ldots,M\right\rbrace,         $
 let $S(a_3,\ldots,a_M)$ denote the set of all $a_2\in\Omega_2$ such that
such that $\inte_{\C^{m_1}} S(\cdot,a_2,a_3,\ldots,a_M    )=\varnothing.$
For $j\in\left\lbrace 4,\ldots,M\right\rbrace,  $
let $S(a_j,\ldots,a_M)$ denote the of all $a_{j-1}\in\Omega_{j-1}$ such that
such that $\Omega_{j-2}\setminus S(a_{j-1},a_j,\ldots,a_M)$ is of Baire category
I, and finally let $\mathcal{S}$  denote the of all $a_{M}\in\Omega_{M}$ such that
such that $\Omega_{M-1}\setminus S(a_M)$ is of Baire category I.
Then $\Omega_M\setminus \mathcal{S}$ is of Baire category I.
\end{lem}
\begin{proof}
 For  $j\in\left\lbrace 1,\ldots,M\right\rbrace$ let $\left(\Q+i\Q\right)^{m_j}=
 \left\lbrace  q^j_1,\ldots,q^j_n,\ldots\right\rbrace$ and
 $\delta_n:=\frac{1}{n},$  $n\in\N.$ For $q\in\Omega_j$ and $r>0,$ let
 $\Delta_q(r)$ denote the polydisc in $\C^{m_j}$ with center $q$ and multi-radius
 $(r,\ldots,r).$

 Suppose in order to get a contradiction that  $\Omega_M\setminus \mathcal{S}$ is of Baire category II.
 Then for all $a_M\in \Omega_M\setminus \mathcal{S},$ $\Omega_{M-1}\setminus S(a_M)$ is of Baire category II.
Therefore, for $j= M-1,\ldots,3  $ and any $a_j\in
  \Omega_{j}\setminus S(a_{j+1},\ldots,a_M), $ the set
 $\Omega_{j-1}\setminus S(a_j,\ldots,a_M)$ is of Baire category II.
Put
\begin{equation*}
S_n(a_3,\ldots,a_M):=\left\lbrace a_2\in\Omega_2:\quad
S(\cdot,a_2,a_3,\ldots,a_M) \supset \Delta_{q^1_n}(\delta_n)\right\rbrace.
\end{equation*}
Since $S$ is relatively closed, $S_n(a_3,\ldots,a_M)$ is also relatively
closed
in $\Omega_2.$ Moreover, from the definition of
$S(a_3,\ldots,a_M),$  we have the following identity
\begin{equation*}
\Omega_2\setminus S(a_3,\ldots,a_M)=
\bigcup_{n=1}^{\infty}S_n(a_3,\ldots,a_M).
\end{equation*}
Since it is shown in the above discussion, that $\Omega_2\setminus
S(a_3,\ldots,a_M)$ is of Baire category II in $\Omega_2,$ we can therefore
apply the Baire Theorem to the right side of the latter identity.
Consequently,
there exist $n_1,n_2\in \N$ such that $ S_{n_1}(a_3,\ldots,a_M) \supset
\Delta_{q^2_{n_2}}(\delta_{n_2}).$ This implies that
$S(\cdot,\cdot,a_3,\ldots,a_M) \supset
\Delta_{q^1_{n_1}}(\delta_{n_1})\times \Delta_{q^2_{n_2}}(\delta_{n_2}).$

Now, define inductively for  $j= 2,\ldots,M-1  $ and $n_1,\ldots,n_{j}\in \N,$
\begin{multline*}
S_{n_1,\ldots,n_j}(a_{j+2},\ldots,a_M)\\:=\left\lbrace a_{j+1}\in\Omega_{j+1}:\
S(\cdots,a_{j+1},\ldots,a_M) \supset \Delta_{q^1_{n_1}}(\delta_{n_1})\times\cdots\times
  \Delta_{q^j_{n_j}}(\delta_{n_j})      \right\rbrace.
\end{multline*}
Since $S$ is relatively closed, $ S_{n_1,\ldots,n_j}(a_{j+2},\ldots,a_M)  $ is also relatively
closed. Moreover, it can be checked that
\begin{equation*}
\Omega_{j+1}\setminus S(a_{j+2},\ldots,a_M) \subset \bigcup_{n_1,\ldots,n_j=1}^{\infty}
S_{n_1,\ldots,n_j}(a_{j+2},\ldots,a_M).
\end{equation*}
Applying the Baire Theorem again, it follows that there are
$n_1,\ldots,n_{j+1}\in\N$ such that
$S_{n_1,\ldots,n_j}(a_{j+2},\ldots,a_M)\supset
\Delta_{q^{j+1}_{n_{j+1}}}(\delta_{n_{j+1}}),$ and hence
\begin{equation*}
S(\cdots,a_{j+2},\ldots,a_M) \supset \Delta_{q^1_{n_1}}(\delta_{n_1})\times\cdots\times
  \Delta_{q^{j+1}_{n_{j+1}}}(\delta_{n_{j+1}}).
  \end{equation*}
Finally, we obtain for $j=M-1$ that $\inte S\not= \varnothing,$
which contradicts the hypothesis. Hence, the proof is  complete.
\end{proof}
\begin{rem}
If we  apply Lemma 4.1  to the case where
$\Omega_1:=D_j$ and $\Omega_2:=   \left (D_1\times \cdots\times D_{j-1} \right )\times
\left (D_{j+1}\times \cdots \times D_{N} \right ). $
Then, for each $j\in\left\lbrace 1,\ldots,N \right\rbrace,$
the set $ S_j$ in the statement of Theorem D is of Baire category I.
In particular, the set $\Omega\setminus ((T\setminus S)\cap\Omega)$ is of Baire
category I for all open sets $\Omega\subset D.$
\end{rem}
\begin{lem}
Let $U\subset \C^p$ and $V\subset \C^q$ be two pseudoconvex domains.
Consider four sets
$C\subset A\subset U$ and $D\subset B\subset V$ such that
$\overline{C}=\overline{A},$  $\overline{D}=\overline{B}$
and $\overline{A},$ $\overline{B}$ are locally pluriregular.
Put $X:=\X(A,B;U,V) $
and $\widehat{X}:=\widehat{\X}(\overline{A},\overline{B};U,V).$
Assume  $f\in\mathcal{O}_s(X)$ and there is a finite constant $K$  such that
for all $c\in C$ and $d\in D,$
\begin{equation*}
\sup_{V} \vert f(c,\cdot)\vert <K\quad \text{and}\quad \sup_{U} \vert f(\cdot,d)\vert
<K.
\end{equation*}
Then there exists a unique function $\hat{f}\in\mathcal{O}(\widehat{X})$
such that
$\hat{f}=f$ on $\widehat{X}\cap X.$
\end{lem}
\begin{proof}
From the hypothesis on the boundedness of $f,$ it follows that the two
families
$\left \lbrace f(c,\cdot):\ c\in C\right\rbrace$ and
$\left \lbrace f(\cdot,d):\ d\in D\right\rbrace$
are normal. We now define two functions $f_1$ on $\overline{A}\times V$
and $f_2$ on $U\times \overline{B}$ as follows.

For any $z\in \overline{A},$ choose a sequence $\left(
c_n\right)_{n=1}^{\infty}\subset C$ such that $\lim_{n\to\infty} c_n=z$
and the sequence $\left(f(c_n,\cdot)\right)_{n=1}^{\infty}$ converges uniformly on compact subsets of $V.$
We let
\begin{equation*}
f_1(z,w):=\lim_{n\to\infty} f(c_n,w),\qquad \text{for all}\ w\in V.
\end{equation*}
Similarly, for any $w\in \overline{B},$ choose a sequence $\left(
d_n\right)_{n=1}^{\infty}\subset D$ such that $\lim_{n\to\infty} d_n=w$
and the sequence $\left(f(\cdot,d_n)\right)_{n=1}^{\infty}$ converges uniformly on compact subsets of $U.$
We let
\begin{equation*}
f_2(z,w):=\lim_{n\to\infty} f(z,d_n),\qquad \text{for all}\ z\in U.
\end{equation*}
We first check that $f_1$ and $f_2$ are well-defined. Indeed, it suffices
to verify this for $f_1$ since the same argument also applies to $f_2.$
Let $\left(
c^{'}_n\right)_{n=1}^{\infty}\subset C$ be another sequence such that $\lim_{n\to\infty} c^{'}_n=z$
and the sequence $\left(f(c_n,\cdot)\right)_{n=1}^{\infty}$ converges uniformly on compact subsets of
$V.$ Since for all $b\in B,$
\begin{equation*}
\lim_{n\to\infty}f(c_n,b)=f(z,b)=\lim_{n\to\infty}f(c^{'}_n,b),
\end{equation*}
and since $B$ is the set of unicity for holomorphic functions on $V,$ our  claim
follows.

One next verifies that $f_1=f_2$ on $\overline{A}\times\overline{B}.$
Indeed, let $z\in\overline{A},\ w\in\overline{B}$ and let $\left(
c_n\right)_{n=1}^{\infty}\subset C,$  $\left(
d_n\right)_{n=1}^{\infty}\subset D$ be as above.
Then clearly, we have
\begin{equation*}
f_1(z,w)=\lim_{n\to\infty}f(c_n,d_n)=f_2(z,w).
\end{equation*}
We are now able to define a function $\tilde{f}$ on $\X(\overline{A},\overline{B};U,V)$
by  the formula $f=f_1$ on $\overline{A}\times V$
and  $f=f_2$ on $U\times\overline{B}.$ It follows from the construction of $f_1$ and
$f_2$ that  $\tilde{f}\in\mathcal{O}_s(  \X(\overline{A},\overline{B};U,V)          ).$

One next checks that $\tilde{f}=f$ on $X.$ Indeed, since for each $a\in
A,$ $f(a,\cdot)$ and $\tilde{f}(a,\cdot)$ are holomorphic, it suffices to
verify that $\tilde{f}(a,d)=f(a,d).$ But the latter equality follows
easily from the definition of $f_1$ and the hypothesis.

Finally, one applies the classical cross theorem (cf. \cite{nz}, \cite{jp1}) to $\tilde{f}\in\mathcal{O}_s
\left(\X(\overline{A},\overline{B};U,V)\right ),$ thus the existence of $\hat{f}$ follows.
The unicity of $\hat{f}$ is also clear.
\end{proof}
\begin{lem} (Rothstein type theorem, cf. \cite{ro}).
Let $f\in \mathcal{O}(E^p\times E^q).$ Assume that $A\subset E^p$ such that
for all open subsets $U\subset E^p,$ $A\cap U$ is
of Baire category II and for
all  $z\in E^p$ we have $(P_f)(z,\cdot)\not=E^q.$ Here $P_f$ denotes the
pole set of $f.$
Let $G\subset \C^q$ be a domain such that $E^q\subset G$
and assume that for all $a\in A,$ the function $f(a,\cdot)$ extends meromorphically to
$\widetilde{f(a,\cdot)} \in \mathcal{M}(G).$
Then for any relatively compact subdomain $\widetilde{G}\subset G,$
there are an open dense set $\mathcal{A}\subset E^p$ and a function
$\tilde{f}\in\mathcal{M}(\Omega),$ where $\Omega:=E^p\times E^q \cup \mathcal{A}\times\widetilde{G}
$ such that $\tilde{f}=f$ on $E^p\times E^q.$ \end{lem}
 \begin{proof} We present a sketch of the proof.

(1) The case where $G:=\Delta_0(R)\ (R>1).$

Arguing as in  the proof of Rothstein's theorem given in \cite{si}, the
conclusion of the lemma follows.

(2) The general case, where $G$ is arbitrary.

Fix an $a\in E^p$ and $r>0.$ Let $B$ denote the set of all $b\in G$ such
that there exist $0<r_b<r,$ an open dense $A_b$ of $\Delta_a(r_b)$ and
$f_b\in\mathcal{M}\left(A_b\times \Delta_b(r_b)\right)$ such
that  for all $\alpha\in A\cap A_b,$ $f_b(\alpha,\cdot)
=\widetilde{f(\alpha,\cdot)}$ on $\Delta_b(r_b).$

Obviously, $B$ is open. Using the case (1) and the hypothesis on $A,$
one can show that $B$ is closed in $G.$ Thus $B=G.$
Moreover, one can also show that if $A_b\cap A_{b^{'}}\not=\varnothing$
and $\Delta_b(r_b)\cap\Delta_{b^{'}}(r_{b^{'}})\not=\varnothing,$
then $f_b=f_{b^{'}}$ on  $\left(A_b\cap A_{b^{'}}\right)\times
\left(\Delta_b(r_b)\cap\Delta_{b^{'}}(r_{b^{'}})\right).$ Therefore,
using the hypothesis that $\widetilde{G}$ is relatively compact,
we see that for any $a\in E^p$ and any $r>0,$ there is an open set
$\mathcal{A}_{a,r}\subset \Delta_a(r)$ and  $f_{a,r}\in\mathcal{M}\left(
\mathcal{A}_{a,r}\times \widetilde{G}\right)$
such that
 for all $\alpha\in A\cap \mathcal{A}_{a,r},$ $f_{a,r}(\alpha,\cdot)
=\widetilde{f(\alpha,\cdot)}$ on $\widetilde{G}.$

Finally, let $\mathcal{A}:=\bigcup_{a\in E^n,\ r>0} \mathcal{A}_{a,r}.$
This open set is clearly dense in $E^p.$  By gluing the function $f_{a,r}$ together, we obtain the desired
meromorphic extension $\tilde{f}\in\mathcal{M}(\Omega);$ so the proof
of the lemma is  completed.
\end{proof}
\section{Proof of Theorem D} 
We will only give the proof of Theorem D for the case where $f$ is
separately meromorphic. Since the case where $f$ is separately holomorphic is
quite similar and in some sense simpler, it is therefore left to the reader.

\smallskip

\noindent{\bf Proof of Part (ii).}

 Put
\begin{equation}
\mathcal{A}_j:=\left\lbrace z_j\in D_j\ :\quad \inte_{\C^{n_j}} S(z_j,\cdot)
=\varnothing\right\rbrace  \qquad\text{for}\ j=1,2.
\end{equation}
By Lemma 4.1,  $D_j\setminus \mathcal{A}_j$ is of Baire category I.
For $a_j\in\mathcal{A}_j\  (j=1,2),$ let $ \widetilde{f(a_1,\cdot)}   $
(resp.  $ \widetilde{f(\cdot,a_2)}   $) denote the meromorphic extension of $f(a_1,\cdot)$
(resp. $f(\cdot,a_2)$) to $D_2$
(resp. to $D_1$).

Let $\mathcal{U}\subset D_1,$ $\mathcal{V}\subset D_2$ be arbitrary open
sets.
 For  a relatively compact pseudoconvex subdomain $V$ of $\mathcal{V}$ and for  a positive number $K,$
   let
$Q^1_{V,K}$ denote the set of $a_1\in \mathcal{A}_1\cap \mathcal{U}$ such that
$\sup_{V}\left\vert \widetilde{f(a_1,\cdot) }  \right\vert\leq K$
(and thus $\widetilde{f(a_1,\cdot)}\in\mathcal{O}(V)$).
By virtue of (5.1)
and the hypothesis,  a countable number of the $Q^1_{V,K}$ cover $\mathcal{A}_1\cap \mathcal{U}.$
Since the latter set is of Baire category II,
we can choose $V,K_1$ such that  the closure $\overline{Q^1_{V,K_1}}$
contains  a polydisc $\Delta_1\subset \mathcal{U}$ and
$Q^1_{V,K_1}\cap\Delta_1 $ is of Baire category II in $\Delta_1.$

For  a relatively compact pseudoconvex subdomain $U$  of $\Delta_1$ and for  a positive number $K,$ we
denote
by $Q^2_{U,K}$  the set of $a_2\in \mathcal{A}_2\cap V$ such that
$\sup_{U}\left\vert \widetilde{f(\cdot,a_2) }  \right\vert\leq K$
(and thus $\widetilde{f(\cdot,a_2)}\in\mathcal{O}(U)$).
By virtue of (5.1) and the hypothesis, a countable number of the $Q^2_{U,K}$ cover $\mathcal{A}_2\cap
  V.$ Since  the latter set is of Baire category II,
we can choose $U,K_2$ such that  $\overline{Q^2_{U,K_2}}$
contains  a polydisc $\Delta_2\subset V$ and
$Q^2_{U,K_2}\cap\Delta_2 $ is of Baire category II in $\Delta_2.$

Now let $K:=\text{max}\{K_1,K_2\}, A:=\mathcal{A}_1\cap U,\ C:=Q^1_{V,K}\cap U,\
B:=\mathcal{A}_2\cap \Delta_2,\ D:=Q^2_{V,K}\cap\Delta_2.$ Then it is easy
to see that $\overline{A}=\overline{C}=U$ and
$\overline{B}=\overline{D}=\Delta_2.$ Moreover, all other hypotheses of
Lemma 4.3 are fulfilled.  Consequently, an application of this lemma
gives the following.

{\it Let $\ \mathcal{U}\subset D_1,$ $\mathcal{V}\subset D_2$ be arbitrary open
sets. Then there is  a polydisc $\Delta_{a}(r)\subset \mathcal{U}\times\mathcal{V}$
 and a function $\hat{f}\in \mathcal{O}(\Delta_{a}(r))$
such that  $\hat{f}=f$ on $(T\setminus S)\cap\Delta_{a}(r).$}

Write $a=(a_1,a_2)\in D_1\times D_2.$ Since the set
$\mathcal{A}_j\cap\Delta_{a_j}(r)$ is of Baire category II and
by replacing $\Delta_{a_j}(r)$ by  a smaller polydisc, we  see that this set
 satisfies the hypothesis of Lemma 4.4.
Consequently, an application of this lemma gives $f^1_a\in
\mathcal{M}(\Delta_{a_1}(r)\times\Omega_2)$ and $f^2_a\in
\mathcal{M}(\Omega_1\times\Delta_{a_2}(r))$ which coincide with $f$
on $(T\setminus S)\cap\Delta_a(r).$ Moreover, one sees
that the function $f_{\mathcal{U},\mathcal{V}}$ given  by
\begin{equation*}
f_{\mathcal{U},\mathcal{V}}:=f^1_a\ \text{on}\
\Delta_{a_1}(r)\times\Omega_2,\ \ \text{and}\
f_{\mathcal{U},\mathcal{V}}:=f^2_a\ \text{on}\
\Omega_1\times\Delta_{a_2}(r),
\end{equation*}
is well-defined, meromorphic on the
cross $X:=\X(\Delta_{a_1}(r),\Delta_{a_2}(r);\Omega_1,\Omega_2),$
and $f_{\mathcal{U},\mathcal{V}}=f $ on $(T\setminus S)\cap X.$

Using Remark 4.2, one can also prove the following. If $\mathcal{U}^{'}\subset D_1,$ $\mathcal{V}^{'}
\subset D_2$ are arbitrary open sets   and $ f_{\mathcal{U}^{'},\mathcal{V}^{'}}$ is the corresponding
meromorphic function defined on the corresponding cross $X^{'},$ then
$f_{\mathcal{U},\mathcal{V}}=f_{\mathcal{U}^{'},\mathcal{V}^{'}}$
on $X\cap X^{'}.$

Let $A_j:=\bigcup_{\mathcal{U}\subset \Omega_1,\ \mathcal{V}\subset
\Omega_2} \Delta_{a_j}(r),$ for $j=1,2.$ It is clear that $A_j$ is an open dense set
in $\Omega_j.$  Then  gluing  all
$f_{\mathcal{U},\mathcal{V}},$ we obtain a function
$\widehat{f}$ meromorphic on $X:=\X(A_1,A_2;\Omega_1,\Omega_2)$
satisfying $\widehat{f}=f $ on $(T\setminus S)\cap X.$
Finally, one applies Theorem 1.3 in \cite{jp4} to $\widehat{f},$
and the conclusion of Part (ii) follows.

\smallskip

\noindent{\bf Proof of Part (i).}

  In the sequel, $\Sigma_M$ will denote the group of
permutation of $M$ elements $\lbrace 1,\ldots,M\rbrace.$ Moreover, for any
$\sigma\in\Sigma_M$ and under the hypothesis and the notation  of Lemma 4.1,
we define
\begin{equation*}
S^{\sigma}:=\left\lbrace z^{\sigma}:\ z\in S\right\rbrace\ \
\text{and}\
\Omega^{\sigma}:=\Omega_{\sigma(1)}\times\cdots\times\Omega_{\sigma(M)},
\end{equation*}
where
\begin{equation*}
z^{\sigma}:=(z_{\sigma(1)},\ldots,z_{\sigma(M)}),\quad\
z\in\Omega=\Omega_1\times\cdots\times\Omega_M.
\end{equation*}
If in the statement of Lemma 4.1, one replaces $S$ by $S^{\sigma}$ and
$\Omega$ by $\Omega^{\sigma},$ then one obtains
$\mathcal{S}^{\sigma},$ $S^{\sigma}(a_{\sigma(N)}),\ldots,S^{\sigma}(a_{\sigma(3)},\ldots,a_{\sigma(N)}).$
The proof will be divided into three steps.\\
 {\bf Step 1: $N=2.$}

By virtue of Part (ii), for each pair of relatively compact
pseudoconvex subdomains $\Omega_j\subset D_j  \  (j=1,2)$
we obtain a polydisc $ \Delta_{\Omega_1,\Omega_2}\subset
\Omega_1\times\Omega_2$ and a function $
f_{\Omega_1,\Omega_2}\in\mathcal{O}(\Delta_{\Omega_1,\Omega_2})$
such that $f= f_{\Omega_1,\Omega_2}$ on $(\Delta_{\Omega_1,\Omega_2}\cap T)\setminus S.$
A routine identity argument shows that every two functions $
f_{\Omega_1,\Omega_2}$ coincide on the intersection of their domains of
definition. Gluing  $
f_{\Omega_1,\Omega_2},$ we  obtain the desired function $\hat{f}\in \mathcal{O} ( \bigcup
  \Delta_{\Omega_1,\Omega_2}       ).$      \\
{\bf  Step 2: $N=3.$}

Consider the following elements of $\Sigma_3.$
\begin{equation*}
\sigma_1:=\left(
\begin{array}{ccc}
1&2&3\\
2&1&3
\end{array}
\right),
\sigma_2:=\left(
\begin{array}{ccc}
1&2&3\\
1&2&3
\end{array}
\right),
\sigma_3:=\left(
\begin{array}{ccc}
1&2&3\\
3&2&1
\end{array}
\right),
\sigma_4:=\left(
\begin{array}{ccc}
1&2&3\\
3&1&2
\end{array}
\right).
\end{equation*}

Fix any subdomain $\Omega_1\times\Omega_2\times\Omega_3\subset D$ and
pick any $a_3\in \mathcal{S}^{\sigma_1}\cap \mathcal{S}^{\sigma_2}.$ Then
by the definition, $\Omega_1\setminus{S}^{\sigma_1}(a_3)$ (resp.
$\Omega_2\setminus{S}^{\sigma_2}(a_3)$) is of Baire category I in
$\Omega_1$ (resp. $\Omega_2$).

Also,
for any $a_1\in S^{\sigma_1}(a_3)\cap \mathcal{S}^{\sigma_3},$ we have
$\inte S(a_1,\cdot,a_3)=\varnothing$ and the set $\Omega_2\setminus
\left\lbrace a_2\in\Omega_2:\ \inte S(a_1,a_2,\cdot)=\varnothing\right\rbrace$
is of Baire category I.

Similarly, for any $a_2\in S^{\sigma_2}(a_3)\cap \mathcal{S}^{\sigma_4},$ we have
$\inte S(\cdot,a_2,a_3)=\varnothing$ and the set $\Omega_1\setminus
\left\lbrace a_1\in\Omega_1:\ \inte S(a_1,a_2,\cdot)=\varnothing\right\rbrace$
is of Baire category I.

Thus $f$ is well-defined on the union $X$ of the two following subsets
of $\Omega_1\times\Omega_2\times\lbrace a_3\rbrace:$
\begin{equation}
\left\lbrace (z_1,z_2,a_3):\quad \text{for any}\ z_1\in S^{\sigma_1}(a_3)\cap
\mathcal{S}^{\sigma_3},\
\text{and}\  z_2\in  S^{\sigma_3}(z_1)\cap S^{\sigma_2}(a_3)\right\rbrace
\end{equation}
and
\begin{equation}
\left\lbrace (z_1,z_2,a_3):\quad \text{for any}\ z_2\in S^{\sigma_2}(a_3)\cap
\mathcal{S}^{\sigma_4},\
\text{and}\  z_1\in  S^{\sigma_4}(z_2)\cap S^{\sigma_1}(a_3)\right\rbrace.
\end{equation}
Observe that by the definition in  Lemma 4.1, $\Omega_1\setminus (  S^{\sigma_4}(z_2)\cap S^{\sigma_1}(a_3)
)$ (resp.  $\Omega_2\setminus (  S^{\sigma_3}(z_1)\cap S^{\sigma_2}(a_3)
) $) is of Baire category I in $\Omega_1$ (resp. $\Omega_2$). By virtue of
(5.2)--(5.3), the same conclusion also holds for the fibers $X(z_1,\cdot,a_3)$
and $X(\cdot,z_2,a_3),$  $z_1\in S^{\sigma_1}(a_3)\cap
\mathcal{S}^{\sigma_3} $ (resp. $ z_2\in S^{\sigma_2}(a_3)\cap
\mathcal{S}^{\sigma_4} $).

Let $\mathcal{U}_j\subset\Omega_j\ (j=1,2) $ be an arbitrary open subset.
If $\Delta:=\Delta_q(r)$ is a polydisc, then we denote by $k\Delta$
the polydisc $\Delta_q(kr)$ for all $k>0.$ Repeating the Baire category argument already used in the proof of Part (ii),
one can show that there are a positive number $K,$  polydiscs
$\Delta_j\subset \mathcal{U}_j,$    and  subsets
$Q^{1}_{ \mathcal{U}_1,\mathcal{U}_2 }$ of  $S^{\sigma_1}(a_3)\cap
\mathcal{S}^{\sigma_3}$
(resp. $Q^{2}_{ \mathcal{U}_1,\mathcal{U}_2 }$ of  $S^{\sigma_2}(a_3)\cap
\mathcal{S}^{\sigma_4} $)
such that $\overline{Q^{j}_{\mathcal{U}_1,\mathcal{U}_2   }}=\Delta_j,$
  $ Q^{j}_{ \mathcal{U}_1,\mathcal{U}_2}$ is of Baire category II,
and $\sup_{2\Delta_2}\left\vert\widetilde{f(a_1,\cdot,a_3)}\right\vert \leq K,$
(resp. $\sup_{2\Delta_1}\left\vert\widetilde{f(\cdot,a_2,a_3)}\right\vert \leq K),$
for $a_j\in Q^{j}_{ \mathcal{U}_1,\mathcal{U}_2}.$

Therefore, by applying Lemma 4.2, we obtain a function
$f_{a_3}=f_{\mathcal{U}_1,\mathcal{U}_2,a_3}\in\mathcal{O}(\Delta_1\times\Delta_2)$
which extends $f(\cdot,\cdot,a_3)$ to $\Delta_1\times\Delta_2\times\lbrace a_3\rbrace$
for all $a_3\in\mathcal{S}^{\sigma_1}\cap\mathcal{S}^{\sigma_2}.$

Now let  $\mathcal{U}_j\subset\Omega_j\ (j=1,2,3) $ be an arbitrary open subset.
Since the set
$\Omega_3\setminus(\mathcal{S}^{\sigma_1}\cap\mathcal{S}^{\sigma_2})$ is of
Baire category I, by using the previous discussion  we are able to perform the Baire category argument
 already used in the proof of Part (ii). Consequently,
 there are  a positive number $K,$  polydiscs
$\Delta_j\subset \mathcal{U}_j,$    and  subsets
$Q^{3}_{ \mathcal{U}_1,\mathcal{U}_2,\mathcal{U}_3 }$ of  $\mathcal{S}^{\sigma_1}\cap\mathcal{S}^{\sigma_2} $
such that $\overline{Q^{3}_{\mathcal{U}_1,\mathcal{U}_2,\mathcal{U}_3   }}=\Delta_3,$
  $ Q^{3}_{ \mathcal{U}_1,\mathcal{U}_2,\mathcal{U}_3  }$ is of Baire category II,
and $\sup_{2\Delta_1\times 2\Delta_2}\left\vert f_{a_3}(\cdot,\cdot)\right\vert \leq K,$
for $a_3\in Q^{3}_{ \mathcal{U}_1,\mathcal{U}_2, \mathcal{U}_3  }.$

By changing the role of $1,2,3$ and by taking smaller polydiscs, we obtain in the same way the
subsets   $ Q^{j}_{ \mathcal{U}_1,\mathcal{U}_2,\mathcal{U}_3  }\subset \Delta_j\ (j=1,2)$
with similar property.

For $j\in\lbrace 1,2,3\rbrace$ consider the following subsets of $T$
\begin{multline*}
T_j:=\left\lbrace a=(a_1,a_2,a_3):\
a_k\in\Delta_k,a_l\in\Delta_l,a_j\in Q^j_{ \mathcal{U}_1,\mathcal{U}_2,
\mathcal{U}_3},\  \lbrace k,l,j\rbrace =\lbrace 1,2,3\rbrace\right.\\
\left.\text{and either}\ \inte_{\C^{n_l}} S(\cdot,a_k,a_j)=\varnothing\  \text{or}\   \inte_{\C^{n_k}}
 S(a_l\cdot,a_j)=\varnothing
\right\rbrace.\end{multline*}
One next proves that
\begin{equation}
f(a)=f_{a_j}(a_k,a_l),\qquad a\in T_1\cup T_2\cup T_3.
\end{equation}
Indeed, let $a=(a_1,a_2,a_3)\in T_3$ with $\inte
S(\cdot,a_2,a_3)=\varnothing.$ In virtue of (5.2), we can choose a sequence
$\left( z_1^n \right)_{n=1}^{\infty}\to a_1$ and for every $n\geq 1$
a sequence $\left( z_2^{m(n)} \right)_{m=1}^{\infty}\to a_2.$ Clearly,
$f(z_1^n,z_2^{m(n)},a_3)=f_{a_3}(z_1^n,z_2^{m(n)}).$
Therefore,
\begin{equation*}
f(a)=\lim_{n\to\infty}f(z_1^n,a_2,a_3)=\lim_{n\to\infty} \lim_{m\to\infty} f(z_1^n,z_2^{m(n)},a_3)
=f_{a_3}(a_1,a_2).
\end{equation*}

Now, we wish to glue the three functions $f_{a_j}\  (j=1,2,3).$
Since the family $\left\lbrace f_{a_j}:\ a_j\in
Q^{j}_{ \mathcal{U}_1,\mathcal{U}_2,\mathcal{U}_3  }\right\rbrace$ is
normal, we define an extension $f_j$ of  $f_{a_j}\  (j=1,2,3)$ to $\Delta:=\Delta_1\times\Delta_2
\times\Delta_3$ as follows.

Let $\lbrace j,k,l\rbrace\in\lbrace 1,2,3\rbrace$ and for
$z=(z_1,z_2,z_3)\in\Delta,$
choose a sequence $\left(a_{j}^n\right)_{n=1}^{\infty} \subset
 Q^j_{ \mathcal{U}_1,\mathcal{U}_2,\mathcal{U}_3  }$ such that
$\lim_{n\to\infty}a_{j}^n=z_j     $
and the sequence $\left(f_{a_j}\right)_{n=1}^{\infty}$ converges uniformly
on compact subsets of $\Delta_k\times\Delta_l.$ We let
\begin{equation}
f_j(z):=\lim_{n\to\infty} f_{a_j^n}(a_k^n,a_l^n),
\end{equation}
for any sequence $\left((a_k^n,a_l^n,a_j^n)\right)_{n=1}^{\infty}\subset
T_j\to z$ as $n\to\infty.$

Let us first check that the functions $f_j$ are well-defined. Indeed, this
assertion will follow from the estimate
\begin{equation}
\left\vert f_{a_j}(a_k,a_l) -f_{b_j}(b_k,b_l) \right\vert
\leq CK\vert a-b\vert,\qquad a=(a_k,a_l,a_j),b=(b_k,b_l,b_j)\in T_j.
\end{equation}
 Here $C$ is a constant that depends only on $\Delta.$
 It now remain to prove (5.6) for example in the case $j=3.$
 To do this, let $z=(z_1,z_2,z_3), w=(w_1,w_2,w_3)\in T_3.$
 Then by virtue of (5.2) and (5.3), one can choose $a_1,a^{'}_1\in\Delta_1$ and $a_2\in\Delta_2$
  such that
 \begin{equation}
 \begin{split}(a_1,a_2,z_3),(a^{'}_1,a_2,z_3)\in T_2\cap T_3,\qquad
  \left\vert z-(a_1,a_2,z_3) \right\vert \leq 2\vert z-w\vert,\\
\vert z-(a^{'}_1,a_2,z_3) \vert \leq 2\vert z-w\vert.
\end{split}\end{equation}
Write
\begin{multline*}
\left\vert f_3(z)-f_3(w)   \right\vert\leq
\left\vert f_3(z)-f_3(a_1,a_2,z_3)   \right\vert
+\left\vert f_3(a_1,a_2,z_3) -f_3(w)  \right\vert\\
+\left\vert f_2(a_1,a_2,z_3)-f_2(a^{'}_1,a_2,z_3)   \right\vert.
\end{multline*}
Since $\sup_{2\Delta_1\times 2\Delta_2}\vert f_{z_3}\vert\leq K,
\sup_{2\Delta_1\times 2\Delta_2}\vert f_{w_3}\vert\leq K$
and $\sup_{2\Delta_1\times\Delta_3}\vert f_{a_2}\vert\leq K,$
applying Schwarz's lemma to the right side of the latter estimate and using  (5.7), the desired estimate (5.6) follows.

From the  construction (5.5) above,
$f_j(\cdot,\cdot,z_j)\in\mathcal{O}(\Delta_k\times\Delta_l).$ Moreover, a
routine identity argument using (5.2) and (5.3) shows that $f_1=f_2=f_3.$
Finally, define
\begin{equation*}
\hat{f}_{ \mathcal{U}_1,\mathcal{U}_2,\mathcal{U}_3  }(z)=f_1(z)=f_2(z)=f_3(z),\qquad z\in\Delta,
\end{equation*}
then $\hat{f}_{ \mathcal{U}_1,\mathcal{U}_2,\mathcal{U}_3  }$ extends $f$ holomorphically from  $T_1\cup T_2\cup T_3$ to
$\Delta.$ A routine identity argument as in (5.?) shows that
$\hat{f}_{ \mathcal{U}_1,\mathcal{U}_2,\mathcal{U}_3  }=f$ on $(T\cap\Delta)\setminus S.$ Gluing $\hat{f}$
for all $\mathcal{U}_1,\mathcal{U}_2,\mathcal{U}_3,$ we obtain the desired
extension function $\hat{f}.$ Hence the proof is complete in this case.
\\
{\bf Step 3: $N\geq 4.$}

The general case uses induction on $N.$ Since the proof is very similar to
the case $N=3$ making use of Lemmas 4.1 and 4.3 and using the inductive hypothesis
for $N-1,$ we leave the details to
the reader. \hfill  $\square$

\end{document}